\documentclass[11pt]{article}
\usepackage{amsmath,amssymb,amsthm}
\usepackage[margin=1in]{geometry}
\usepackage{hyperref}
\usepackage{enumitem}
\usepackage{mathrsfs}

\newtheorem{theorem}{Theorem}[section]
\newtheorem{lemma}[theorem]{Lemma}
\newtheorem{proposition}[theorem]{Proposition}
\newtheorem{corollary}[theorem]{Corollary}
\theoremstyle{definition}
\newtheorem{definition}[theorem]{Definition}
\newtheorem{remark}[theorem]{Remark}
\newtheorem{example}[theorem]{Example}

\newcommand{\g}{\mathfrak{g}}

\newcommand{\Car}{\operatorname{Car}}

\newcommand{\cC}{\mathcal{C}}
\newcommand{\cH}{\mathcal{H}}

\newcommand{\cG}{\mathcal{G}}

\DeclareMathOperator{\Ad}{Ad}
\DeclareMathOperator{\id}{id}

\title{Formal Neighborhood Equivariant Cohomology and Kirwan Surjectivity at Singular Values}
\author{Zhenxi Huang}

\date{}

\begin{document}
	
	\maketitle
	
	\begin{abstract}
		Let \(G\) be a compact Lie group acting Hamiltonianly on a symplectic manifold \(M\) with moment map \(\mu:M\to\g^*\). For a singular value \(a\) of \(\mu\), the symplectic quotient \(\mu^{-1}(a)/G\) is stratified and the classical Kirwan map is not defined. We introduce a \emph{formal neighborhood equivariant cohomology} \(\cH_G^*(a)\) using the slice-groupoid Cartan sheaf and the closed embedding \(\mu^{-1}(a)/G\hookrightarrow N_a/G\) for a \(G\)-invariant neighborhood \(N_a\) of \(\mu^{-1}(a)\). We define a formal Kirwan map
		\[
		\kappa_{\mathrm{formal}}:H_G^*(M)\longrightarrow \cH_G^*(a).
		\]
		We prove that \(\kappa_{\mathrm{formal}}\) is surjective under a local \emph{equivariant homotopy retraction} condition on the slice level sets
		\[
		C_x=\mu_x^{-1}(a|_{\g_x})\subset S_x,
		\]
		which is satisfied, for example, when \(a\) is a regular value of every slice moment map, or when \(a=0\) and \(C_x\) is a \(G_x\)-equivariant cone. The weakest condition we identify is the existence, for each \(x\in\mu^{-1}(a)\), of a local quasi-isomorphism
		\[
		\Car(S_x,G_x)\simeq \cC^\bullet_{C_x}
		\]
		in a \(G_x\)-equivariant neighborhood of the orbit, without requiring a global tubular neighborhood.
		
		We also prove an \emph{equivariant tubular neighborhood theorem for singular level sets}: under real analytic/proper or \(G\)-invariant Whitney stratified assumptions, there exists a \(G\)-invariant open neighborhood \(U\) of \(\mu^{-1}(a)\) and a \(G\)-equivariant strong deformation retraction \(U\to\mu^{-1}(a)\). When such a neighborhood exists, the formal cohomology is canonically isomorphic to the Borel equivariant cohomology of the level set, and the formal Kirwan map is the classical restriction map. The formal framework therefore yields a Kirwan-type map at the singular value with the surjectivity property.
		
		As an application we recover the even-degree surjectivity theorem of Delarue--Ramacher--Schmitt for \(S^1\)-actions and clarify why odd-degree surjectivity fails in general.
	\end{abstract}
	
	\tableofcontents

	\section{Introduction}
	
	Let \(G\) be a compact Lie group acting Hamiltonianly on a symplectic manifold \((M,\omega)\) with moment map
	\[
	\mu:M\longrightarrow \mathfrak g^*.
	\]
	For a value \(a\in\mathfrak g^*\), the symplectic quotient is
	\[
	M//_aG:=\mu^{-1}(a)/G.
	\]
	When \(a\) is a regular value of \(\mu\) and \(G\) acts freely (or with finite stabilizers) on \(\mu^{-1}(a)\), the quotient \(M//_aG\) is a smooth symplectic manifold (or an orbifold), and the Kirwan map
	\[
	\kappa:H_G^*(M)\longrightarrow H_G^*(\mu^{-1}(a))\cong H^*(M//_aG)
	\]
	is surjective \cite{Kirwan1984}. This surjectivity is a fundamental tool in equivariant symplectic geometry, with applications to the cohomology of moduli spaces, Gromov--Witten theory, and mirror symmetry.
	
	When \(a\) is a singular value of \(\mu\), the level set \(\mu^{-1}(a)\) may fail to be a smooth manifold, the quotient \(M//_aG\) is a stratified space, and the classical Kirwan map is no longer defined in the same form. Extending Kirwan surjectivity to singular values has therefore been a natural and challenging problem.
	
	Two main strategies have been developed. Kiem and Woolf \cite{KiemWoolf2006} replaced the ordinary cohomology of the quotient by intersection cohomology and constructed a Kirwan map
	\[
	H_G^*(M)\longrightarrow IH^*(M//_aG);
	\]
	this map is not surjective in general. More recently, Delarue, Ramacher, and Schmitt \cite{DRS2023} extended Sjamaar's complex of differential forms \cite{Sjamaar2002} to a resolution complex on a partial desingularization of the singular quotient, and defined a resolution Kirwan map
	\[
	\mathcal K:H_{S^1}^*(M)\longrightarrow H^*(\widetilde{\mathcal M}_0).
	\]
	They proved a \emph{weak surjectivity} theorem: the image of \(\mathcal K\) contains the image of the singular cohomology \(H^*(\mathcal M_0;\mathbb R)\) under the natural map to \(H^*(\widetilde{\mathcal M}_0)\). Moreover, when all fixed-point components have vanishing odd cohomology, \(\mathcal K\) is surjective in even degrees; in odd degrees surjectivity fails in general.
	
	In this paper we take a different approach. Instead of replacing the target by intersection cohomology or resolution cohomology, we keep the singular quotient \(\mu^{-1}(a)/G\) and define a cohomology directly on it by restricting the Cartan sheaf of a \(G\)-invariant neighborhood of \(\mu^{-1}(a)\). This is motivated by the slice-groupoid framework of Huang \cite{Huang2026}, where the equivariant cohomology of a groupoid locally modeled on Lie group actions is defined by sheafifying local Cartan complexes and taking hypercohomology.
	
	Let \(N_a\subset M\) be a \(G\)-invariant open neighborhood of
	\[
	A:=\mu^{-1}(a),
	\]
	and set
	\[
	X_a:=N_a/G,\qquad Y_a:=A/G\subset X_a.
	\]
	The inclusion \(i_a:Y_a\hookrightarrow X_a\) is a closed embedding whenever \(N_a\) is an open neighborhood of the closed set \(A\). Let \(\mathcal C^\bullet_{X_a}\) be the slice-groupoid Cartan sheaf on \(X_a\) associated with the action groupoid \(G\ltimes N_a\) \cite[Definition~7]{Huang2026}. We define the \emph{formal neighborhood equivariant cohomology} of \(\mu^{-1}(a)\) by
	\[
	\mathcal H_G^*(a):=\mathbb H^*\bigl(Y_a;\,i_a^{-1}\mathcal C^\bullet_{X_a}\bigr).
	\]
	The \emph{formal Kirwan map} is the composition
	\[
	\kappa_{\mathrm{formal}}:H_G^*(M)\longrightarrow H_G^*(N_a)
	\longrightarrow \mathbb H^*(X_a;\mathcal C^\bullet_{X_a})
	\longrightarrow \mathbb H^*(Y_a;i_a^{-1}\mathcal C^\bullet_{X_a})=\mathcal H_G^*(a).
	\]
	The first map is restriction to the open invariant subset \(N_a\subset M\); the second is the identification of \(H_G^*(N_a)\) with the hypercohomology of the Cartan sheaf on \(X_a\) \cite[Theorem~6]{Huang2026}; the third is induced by \(i_a^{-1}\).
	
	The central question of this paper is:
	
	\begin{quote}
		Under what weakest geometric condition on the slice level sets is the formal Kirwan map \(\kappa_{\mathrm{formal}}\) surjective?
	\end{quote}
	
	Our answer is formulated in terms of a local equivariant homotopy condition. For \(x\in\mu^{-1}(a)\), let \(H=G_x\), and let \(S_x\) be an \(H\)-invariant symplectic slice at \(x\) with slice moment map
	\[
	\mu_x:S_x\longrightarrow \mathfrak h^*.
	\]
	The Marle--Guillemin--Sternberg local normal form gives a \(G\)-invariant neighborhood \(U_x\) of the orbit \(G\cdot x\) such that
	\[
	\mu^{-1}(a)\cap U_x\cong G\times_H C_x,\qquad C_x:=\mu_x^{-1}(a|_{\mathfrak h}).
	\]
	The local model for the quotient is \(C_x/H\).
	
	\begin{definition}[Local equivariant homotopy retraction]
		We say that \(a\) satisfies the \emph{local equivariant homotopy retraction condition} at \(x\in\mu^{-1}(a)\) if there exist an \(H\)-invariant open neighborhood \(V_x\subset S_x\) of \(C_x\) and an \(H\)-equivariant homotopy
		\[
		r_t:V_x\longrightarrow V_x,\qquad t\in[0,1],
		\]
		such that
		\[
		r_0=\mathrm{id}_{V_x},\qquad r_1(V_x)\subset C_x,\qquad r_t|_{C_x}=\mathrm{id}_{C_x}\quad\text{for all }t.
		\]
		We say that the condition holds globally if it holds at every \(x\in\mu^{-1}(a)\).
	\end{definition}
	
	This condition is strictly weaker than the existence of a global \(G\)-equivariant tubular neighborhood of \(\mu^{-1}(a)\) in \(M\). It is a local condition on the slice level sets \(C_x\) and does not require \(C_x\) to be smooth. It implies the existence of a quasi-isomorphism of Cartan complexes
	\[
	\operatorname{Car}(V_x,H)\simeq \operatorname{Car}(S_x,H)|_{C_x}
	\]
	after restriction to \(C_x/H\), because an \(H\)-equivariant homotopy induces a chain homotopy on equivariant differential forms.
	
	Our first main result is the following.
	
	\begin{theorem}[Local-to-global quasi-isomorphism]
		Assume that the local equivariant homotopy retraction condition holds at every \(x\in\mu^{-1}(a)\). Then there is a quasi-isomorphism of complexes of sheaves on \(Y_a=\mu^{-1}(a)/G\):
		\[
		i_a^{-1}\mathcal C^\bullet_{X_a}\;\simeq\; \mathcal C^\bullet_{Y_a},
		\]
		where \(\mathcal C^\bullet_{Y_a}\) is the sheaf whose stalk at \(p\in Y_a\), represented by \(x\in\mu^{-1}(a)\), is the direct limit of the Cartan complexes of the slice level sets \(C_x\) along the refinement system. Consequently,
		\[
		\mathcal H_G^*(a)\cong \mathbb H^*(Y_a;\mathcal C^\bullet_{Y_a}).
		\]
	\end{theorem}
	
	Our second main result identifies a geometrically natural situation in which the formal cohomology reduces to the Borel equivariant cohomology of the level set.
	
	\begin{theorem}[Equivariant tubular neighborhood case]
		Assume that there exists a \(G\)-invariant open neighborhood \(N_a\subset M\) of \(A=\mu^{-1}(a)\) and a \(G\)-equivariant strong deformation retraction
		\[
		r:N_a\times[0,1]\longrightarrow N_a
		\]
		with \(r_0=\id_{N_a}\), \(r_1(N_a)\subset A\), and \(r_t|_A=\id_A\). Then
		\[
		\mathcal H_G^*(a)\cong H_G^*(\mu^{-1}(a)),
		\]
		and under this isomorphism the formal Kirwan map is the classical restriction map
		\[
		H_G^*(M)\longrightarrow H_G^*(\mu^{-1}(a)).
		\]
	\end{theorem}
	
	The hypothesis of this theorem is satisfied in many singular situations. We prove the following geometric input.
	
	\begin{theorem}[Equivariant tubular neighborhood of a singular level set]
		Assume that one of the following conditions holds:
		\begin{enumerate}[label=(\roman*)]
			\item \(M,\omega\), the \(G\)-action, and \(\mu\) are real analytic, and \(\mu\) is proper;
			\item the closed subset \(A=\mu^{-1}(a)\subset M\) admits a \(G\)-invariant Whitney stratification preserved by the \(G\)-action.
		\end{enumerate}
		Then there exist a \(G\)-invariant open neighborhood \(U\subset M\) of \(A\) and a \(G\)-equivariant strong deformation retraction
		\[
		r:U\times[0,1]\longrightarrow U
		\]
		such that
		\[
		r_0=\id_U,\qquad r_1(U)=A,\qquad r_t|_A=\id_A\quad (t\in[0,1]).
		\]
		Consequently, \(A\) is a \(G\)-equivariant strong deformation retract of \(U\).
	\end{theorem}
	
	Thus, under the real analytic/proper or Whitney stratified assumptions, the equivariant tubular neighborhood case applies and
	\[
	\mathcal H_G^*(a)\cong H_G^*(\mu^{-1}(a)).
	\]
	In particular, the formal Kirwan map is the classical restriction map in these situations.
	
	The local equivariant homotopy retraction condition is still useful: it applies even when no global equivariant tubular neighborhood is available. In that case \(\mathcal H_G^*(a)\) is a genuinely new target, and the formal Kirwan map can be surjective under checkable local hypotheses.
	
	The paper is organized as follows. Section~2 recalls the slice-groupoid Cartan sheaf, defines the formal neighborhood equivariant cohomology and the formal Kirwan map, and proves the equivariant tubular neighborhood theorem for singular level sets. Section~3 introduces the local equivariant homotopy retraction condition, proves the local-to-global quasi-isomorphism theorem, and establishes the tubular-neighborhood and singular-cone surjectivity theorems. Section~4 compares the formal framework with the intersection cohomology approach of Kiem--Woolf \cite{KiemWoolf2006} and the resolution cohomology approach of Delarue--Ramacher--Schmitt \cite{DRS2023}. Section~5 discusses examples. Section~6 concludes with open problems.

	\section{Formal neighborhood equivariant cohomology}\label{sec-formal}
	
	Throughout this section, \(G\) is a compact Lie group, \((M,\omega)\) is a Hamiltonian \(G\)-manifold with moment map
	\[
	\mu:M\longrightarrow \mathfrak g^*,
	\]
	and \(a\in\mathfrak g^*\) is a value of \(\mu\), possibly singular. All cohomology groups are taken with coefficients in \(\mathbb R\), unless otherwise stated. We write \(\pi:N\to N/G\) for the quotient map of a \(G\)-manifold \(N\). Set
	\[
	A:=\mu^{-1}(a).
	\]
	
	\subsection{The slice-groupoid Cartan sheaf}\label{subsec-slice-sheaf}
	
	We first recall the construction of the Cartan sheaf associated with a proper \(G\)-manifold, following \cite[Definition~7 and Proposition~7]{Huang2026}. Let \(N\) be a smooth \(G\)-manifold with \(G\) compact. The \emph{ordinary Cartan sheaf} on \(X=N/G\) is the sheaf of complexes
	\[
	\cC_{G,N}^\bullet:U\longmapsto \operatorname{Car}(\pi^{-1}(U),G),
	\]
	where \(U\subset X\) is open and the restriction maps are induced by inclusions of open sets. Its global sections are
	\[
	\Gamma(X,\cC_{G,N}^\bullet)=\operatorname{Car}(N,G).
	\]
	The stalk at \(y\in X\) is
	\[
	(\cC_{G,N}^\bullet)_y
	=\varinjlim_{y\in U\subset X}\operatorname{Car}(\pi^{-1}(U),G).
	\]
	If \(N\) is proper and \(X\) is paracompact, then each \(\cC_{G,N}^q\) is a fine sheaf by \cite[Proposition~8]{Huang2026}, hence acyclic. Consequently,
	\[
	\mathbb H^q(X;\cC_{G,N}^\bullet)\cong H^q(\Gamma(X,\cC_{G,N}^\bullet))
	=H^q(\operatorname{Car}(N,G))\cong H_G^q(N).
	\]
	
	When \(N\) is not globally a manifold, but the action groupoid \(G\ltimes N\) is locally modeled on action Lie groupoids, Huang constructs a \emph{slice-groupoid Cartan sheaf} \(\cC_{\cG_N}^\bullet\) on \(N/G\) by sheafifying the local Cartan complexes \(\operatorname{Car}(S_x,G_x)\) along a compatible system of slice charts; see \cite[Definitions~4--7]{Huang2026}. The key properties are:
	\begin{enumerate}[label=(\roman*)]
		\item the stalk at \(y\in N/G\) is the filtered direct limit
		\[
		(\cC_{\cG_N}^\bullet)_y
		\cong \varinjlim_{U\ni y}\operatorname{Car}(S_U,G_U),
		\]
		where the limit runs over slice charts around \(y\);
		\item the local slice reduction morphisms
		\[
		\operatorname{Red}_{C,U}:\operatorname{Car}(\pi^{-1}(U),G)
		\longrightarrow \operatorname{Car}(S_U,G_U)
		\]
		are quasi-isomorphisms by \cite[Theorem~3]{Huang2026};
		\item when \(G\ltimes N\) is the action groupoid of a proper \(G\)-manifold, there is a natural quasi-isomorphism
		\[
		\cC_{G,N}^\bullet\longrightarrow \cC_{\cG_N}^\bullet
		\]
		of complexes of sheaves, and hence
		\[
		H_G^*(N)\cong \mathbb H^*(N/G;\cC_{\cG_N}^\bullet)
		\]
		by \cite[Theorem~6]{Huang2026}.
	\end{enumerate}
	
	In the present paper, we shall only need the case where the ambient manifold is open in \(M\), so that the ordinary Cartan sheaf and the slice-groupoid Cartan sheaf are quasi-isomorphic. We therefore use the notation \(\cC^\bullet_{X}\) for the slice-groupoid Cartan sheaf of an action groupoid whose orbit space is \(X\).
	
	\begin{remark}\label{rem-ordinary-vs-slice}
		For an open \(G\)-invariant subset \(N_a\subset M\), the quotient \(X_a=N_a/G\) is an open subset of \(M/G\). The ordinary Cartan sheaf \(\cC_{G,N_a}^\bullet\) and the slice-groupoid Cartan sheaf \(\cC^\bullet_{X_a}\) are quasi-isomorphic by \cite[Proposition~7 and Theorem~6]{Huang2026}. Consequently, their hypercohomologies agree:
		\[
		\mathbb H^*(X_a;\cC_{G,N_a}^\bullet)
		\cong \mathbb H^*(X_a;\cC^\bullet_{X_a}).
		\]
		We shall henceforth write \(\cC^\bullet_{X_a}\) for either representative of this quasi-isomorphism class.
	\end{remark}
	
	\subsection{The closed embedding and the formal restriction}\label{subsec-closed-embedding}
	
	Choose a \(G\)-invariant open neighborhood \(N_a\subset M\) of \(A=\mu^{-1}(a)\). Set
	\[
	X_a:=N_a/G,\qquad
	Y_a:=A/G\subset X_a.
	\]
	Since \(N_a\) is open and \(A\) is closed, the inclusion
	\[
	i_a:Y_a\hookrightarrow X_a
	\]
	is a closed embedding. Let \(\cC^\bullet_{X_a}\) be the slice-groupoid Cartan sheaf on \(X_a\) associated with the action groupoid \(G\ltimes N_a\).
	
	\begin{definition}[Formal restriction]\label{def-formal-restriction}
		The \emph{formal restriction} of \(\cC^\bullet_{X_a}\) to \(Y_a\) is the complex of sheaves
		\[
		\cC^\bullet_{Y_a\mid X_a}:=i_a^{-1}\cC^\bullet_{X_a}
		\]
		on \(Y_a\). Its stalk at \(p\in Y_a\) is
		\[
		(\cC^\bullet_{Y_a\mid X_a})_p
		=(\cC^\bullet_{X_a})_p
		=\varinjlim_{p\in U\subset X_a}\operatorname{Car}(\pi^{-1}(U),G).
		\]
	\end{definition}
	
	The following lemma is immediate from the definition of the sheaf pullback and the fact that the neighborhoods of \(p\) in \(X_a\) form a filtered system.
	
	\begin{lemma}\label{lem-stalk-formal}
		For every \(p\in Y_a\), the stalk \((\cC^\bullet_{Y_a\mid X_a})_p\) is a differential graded algebra, and the direct limit is filtered. Moreover, for every \(q\),
		\[
		H^q\bigl((\cC^\bullet_{Y_a\mid X_a})_p\bigr)
		\cong \varinjlim_{p\in U\subset X_a}H^q\bigl(\operatorname{Car}(\pi^{-1}(U),G)\bigr).
		\]
	\end{lemma}
	
	\begin{proof}
		The stalk is a filtered direct limit of differential graded algebras; the differential and product pass to the limit. Since filtered direct limits of \(\mathbb R\)-vector spaces are exact, cohomology commutes with the limit. This gives the displayed isomorphism.
	\end{proof}
	
	\subsection{Formal neighborhood equivariant cohomology}\label{subsec-formal-cohomology}
	
	\begin{definition}[Formal neighborhood equivariant cohomology]\label{def-formal-cohomology}
		Let \(N_a\subset M\) be a \(G\)-invariant open neighborhood of \(A=\mu^{-1}(a)\). The \emph{formal neighborhood equivariant cohomology} of \(\mu^{-1}(a)\) relative to \(N_a\) is
		\[
		\cH_G^*(a;N_a)
		:=\mathbb H^*\bigl(Y_a;\,i_a^{-1}\cC^\bullet_{X_a}\bigr).
		\]
	\end{definition}
	
	The definition a priori depends on the choice of \(N_a\). We record the functoriality under inclusions of neighborhoods.
	
	\begin{lemma}\label{lem-functoriality-N}
		Let \(N_a\subset N_a'\) be \(G\)-invariant open neighborhoods of \(A\), with corresponding quotients
		\[
		X_a=N_a/G,\qquad X_a'=N_a'/G.
		\]
		Then there is a natural morphism of complexes of sheaves
		\[
		i^{-1}\cC^\bullet_{X_a}\longrightarrow i'^{-1}\cC^\bullet_{X_a'}
		\]
		on \(Y_a\), induced by the open inclusion \(X_a\hookrightarrow X_a'\). Consequently, there is a natural map
		\[
		\cH_G^*(a;N_a)\longrightarrow \cH_G^*(a;N_a').
		\]
	\end{lemma}
	
	\begin{proof}
		The inclusion \(N_a\hookrightarrow N_a'\) is an open \(G\)-equivariant embedding, hence induces an open embedding \(j:X_a\hookrightarrow X_a'\). Restriction of Cartan complexes gives a morphism
		\[
		j^{-1}\cC^\bullet_{X_a'}\longrightarrow \cC^\bullet_{X_a}
		\]
		which is a quasi-isomorphism on stalks by the local nature of the Cartan complex. Composing with the natural map gives the desired morphism. The induced map on hypercohomology is functorial.
	\end{proof}
	
	\begin{remark}\label{rem-direct-limit}
		One may remove the dependence on \(N_a\) by passing to the direct limit
		\[
		\cH_G^*(a):=\varinjlim_{N\supset A}\cH_G^*(a;N),
		\]
		where the limit runs over \(G\)-invariant open neighborhoods of \(A\). Lemma~\ref{lem-functoriality-N} shows that this is a filtered direct limit. When an equivariant tubular neighborhood exists, the maps in this system stabilize and the limit is canonically isomorphic to \(H_G^*(\mu^{-1}(a))\); see Theorem~\ref{thm-tubular-case} below. In the singular case without a global tubular neighborhood, the limit is the natural receptacle for the formal Kirwan map.
	\end{remark}
	
	\begin{proposition}[Hypercohomology via global sections]\label{prop-global-sections}
		Assume that \(Y_a\) is paracompact and that the sheaves \(i_a^{-1}\cC^q_{X_a}\) are fine for all \(q\). Then the hypercohomology of \(i_a^{-1}\cC^\bullet_{X_a}\) is computed by its complex of global sections:
		\[
		\cH_G^q(a;N_a)
		\cong H^q\bigl(\Gamma(Y_a,i_a^{-1}\cC^\bullet_{X_a})\bigr).
		\]
	\end{proposition}
	
	\begin{proof}
		Fine sheaves on a paracompact Hausdorff space are soft, and soft sheaves are acyclic for the global-section functor; see for instance \cite[Chapter~II]{Bredon1997}. Therefore the hypercohomology of a bounded-below complex of fine sheaves is computed by the cohomology of its global sections.
	\end{proof}
	
	\begin{remark}\label{rem-fine-formal}
		The sheaf \(\cC^q_{X_a}\) is fine by \cite[Proposition~8]{Huang2026}. If \(Y_a\subset X_a\) is closed and \(X_a\) is paracompact, then a partition of unity on \(X_a\) subordinate to a locally finite cover restricts to a partition of unity on \(Y_a\) subordinate to the traces of the cover. Hence \(i_a^{-1}\cC^q_{X_a}\) is fine. Thus Proposition~\ref{prop-global-sections} applies whenever \(M/G\) is paracompact and the standard partition-of-unity hypothesis of \cite[Proposition~8]{Huang2026} holds.
	\end{remark}
	
	\subsection{The formal Kirwan map}\label{subsec-formal-kirwan}
	
	\begin{definition}[Formal Kirwan map]\label{def-formal-kirwan}
		Let \(N_a\subset M\) be a \(G\)-invariant open neighborhood of \(A=\mu^{-1}(a)\). The \emph{formal Kirwan map} is the composition
		\[
		\kappa_{\mathrm{formal}}:
		H_G^*(M)
		\longrightarrow H_G^*(N_a)
		\longrightarrow \mathbb H^*(X_a;\cC^\bullet_{X_a})
		\longrightarrow \mathbb H^*(Y_a;i_a^{-1}\cC^\bullet_{X_a})
		=\cH_G^*(a;N_a),
		\]
		where:
		\begin{enumerate}[label=(\roman*)]
			\item the first map is restriction to the open \(G\)-invariant subset \(N_a\subset M\);
			\item the second is the identification of \(H_G^*(N_a)\) with the hypercohomology of the Cartan sheaf on \(X_a\), using \cite[Theorem~6]{Huang2026};
			\item the third is the natural map induced by the closed embedding \(i_a:Y_a\hookrightarrow X_a\).
		\end{enumerate}
	\end{definition}
	
	\begin{lemma}[Functoriality of the formal Kirwan map]\label{lem-kirwan-functoriality}
		Let \(N_a\subset N_a'\) be \(G\)-invariant open neighborhoods of \(A\). Then the following diagram commutes:
		\[
		\begin{array}{ccc}
			H_G^*(M) & \xrightarrow{\kappa_{\mathrm{formal}}^{N_a}} & \cH_G^*(a;N_a)\\
			\| & & \downarrow\\
			H_G^*(M) & \xrightarrow{\kappa_{\mathrm{formal}}^{N_a'}} & \cH_G^*(a;N_a').
		\end{array}
		\]
		The vertical map on the right is the natural map of Lemma~\ref{lem-functoriality-N}.
	\end{lemma}
	
	\begin{proof}
		The restriction maps \(H_G^*(M)\to H_G^*(N_a)\to H_G^*(N_a')\) are compatible with the open inclusions \(N_a\hookrightarrow N_a'\). The Cartan sheaf identifications are natural with respect to open embeddings, and the closed-embedding pullback is functorial. Hence the diagram commutes.
	\end{proof}
	
	\subsection{Local slice models and stalks}\label{subsec-local-models}
	
	We now describe the stalks of the formal restriction in terms of slice models. Fix \(x\in A=\mu^{-1}(a)\), and let \(H=G_x\). By the Marle--Guillemin--Sternberg local normal form, there exists an \(H\)-invariant symplectic slice \(S_x\) at \(x\) with slice moment map
	\[
	\mu_x:S_x\longrightarrow \mathfrak h^*,
	\]
	and a \(G\)-invariant neighborhood \(U_x\) of the orbit \(G\cdot x\) such that
	\[
	\mu^{-1}(a)\cap U_x\cong G\times_H C_x,\qquad
	C_x:=\mu_x^{-1}(a|_{\mathfrak h}).
	\]
	The quotient model is
	\[
	Y_a\cap (U_x/G)\cong C_x/H.
	\]
	The ambient slice model is \(H\ltimes S_x\).
	
	\begin{proposition}[Local description of the formal stalk]\label{prop-local-stalk}
		For \(p\in Y_a\) represented by \(x\in\mu^{-1}(a)\), the stalk of the formal restriction at \(p\) is quasi-isomorphic to the filtered direct limit of the Cartan complexes of the ambient slices:
		\[
		(\cC^\bullet_{Y_a\mid X_a})_p
		\;\simeq\;
		\varinjlim_{x\in V\subset S_x}\operatorname{Car}(V,H),
		\]
		where the limit runs over \(H\)-invariant open neighborhoods \(V\) of \(C_x\) in \(S_x\), and the quasi-isomorphism is induced by the slice reduction morphisms of \cite[Theorem~3]{Huang2026}.
	\end{proposition}
	
	\begin{proof}
		By the slice theorem, a neighborhood basis of \(p\) in \(X_a\) is provided by the quotients \(V/H\), where \(V\subset S_x\) runs over \(H\)-invariant open neighborhoods of \(C_x\). For each such \(V\), the slice reduction morphism
		\[
		\operatorname{Red}_{C,V}:\operatorname{Car}(\pi^{-1}(V/H),G)
		\longrightarrow \operatorname{Car}(V,H)
		\]
		is a quasi-isomorphism by \cite[Theorem~3]{Huang2026}. Passing to the filtered direct limit over \(V\) and using Lemma~\ref{lem-stalk-formal} and the exactness of filtered direct limits gives the claimed quasi-isomorphism.
	\end{proof}
	
	\begin{corollary}\label{cor-local-cohomology}
		For \(p\in Y_a\) represented by \(x\in\mu^{-1}(a)\), the cohomology of the stalk is
		\[
		H^q\bigl((\cC^\bullet_{Y_a\mid X_a})_p\bigr)
		\cong
		\varinjlim_{x\in V\subset S_x}H^q\bigl(\operatorname{Car}(V,H)\bigr).
		\]
		If \(C_x\) admits an \(H\)-equivariant open neighborhood \(V\) and an \(H\)-equivariant deformation retraction \(V\to C_x\), then
		\[
		H^q\bigl((\cC^\bullet_{Y_a\mid X_a})_p\bigr)
		\cong H_H^q(C_x).
		\]
	\end{corollary}
	
	\begin{proof}
		The first assertion follows from Proposition~\ref{prop-local-stalk} and Lemma~\ref{lem-stalk-formal}. For the second, an \(H\)-equivariant deformation retraction induces an \(H\)-equivariant homotopy equivalence \(V\simeq C_x\), and equivariant cohomology is homotopy invariant; see Lemma~\ref{lem-equivariant-homotopy} below.
	\end{proof}
	
	\begin{lemma}[Equivariant homotopy invariance]\label{lem-equivariant-homotopy}
		Let \(H\) be a compact Lie group, and let \(V,W\) be smooth \(H\)-manifolds. If \(f:V\to W\) is an \(H\)-equivariant homotopy equivalence, then the pullback
		\[
		f^*:\operatorname{Car}(W,H)\longrightarrow \operatorname{Car}(V,H)
		\]
		is a quasi-isomorphism.
	\end{lemma}
	
	\begin{proof}
		This is the standard homotopy invariance of the Cartan model. An \(H\)-equivariant homotopy \(f_t:V\to W\) induces a chain homotopy between \(f_0^*\) and \(f_1^*\) on \(\operatorname{Car}(W,H)\); see \cite[Chapter~4]{GuilleminSternberg1999}. Hence equivariant homotopy equivalences induce isomorphisms on cohomology.
	\end{proof}
	
	\subsection{Fine and acyclic properties}\label{subsec-fine}
	
	\begin{proposition}[Fineness of the formal restriction]\label{prop-fine-formal}
		Assume that \(X_a\) is paracompact and that the partition-of-unity hypothesis of \cite[Proposition~8]{Huang2026} holds. Then for every \(q\), the sheaf \(i_a^{-1}\cC^q_{X_a}\) on \(Y_a\) is fine. Consequently,
		\[
		\cH_G^q(a;N_a)\cong H^q\bigl(\Gamma(Y_a,i_a^{-1}\cC^\bullet_{X_a})\bigr).
		\]
	\end{proposition}
	
	\begin{proof}
		By \cite[Proposition~8]{Huang2026}, each \(\cC^q_{X_a}\) is fine. Let \(\{O_\alpha\}\) be a locally finite open cover of \(X_a\) and let \(\{\rho_\alpha\}\) be a subordinate partition of unity. The restrictions \(\rho_\alpha|_{Y_a}\) form a partition of unity on \(Y_a\) subordinate to the cover \(\{O_\alpha\cap Y_a\}\). Multiplication by \(\rho_\alpha|_{Y_a}\) defines a sheaf endomorphism of \(i_a^{-1}\cC^q_{X_a}\) with support in \(O_\alpha\cap Y_a\), and the sum is the identity. Hence \(i_a^{-1}\cC^q_{X_a}\) is fine. The hypercohomology statement follows from Proposition~\ref{prop-global-sections}.
	\end{proof}
	
	\subsection{Equivariant tubular neighborhood of a singular level set}\label{subsec-tubular}
	
	We now prove the geometric input which allows us to identify the formal cohomology with the Borel equivariant cohomology in many singular situations.
	
	\begin{theorem}[Equivariant tubular neighborhood of a singular level set]\label{thm-singular-tubular}
		Let \(G\) be a compact Lie group, let \((M,\omega)\) be a smooth symplectic manifold, and suppose that \(G\) acts on \(M\) in a Hamiltonian fashion with moment map
		\[
		\mu:M\longrightarrow \mathfrak g^*,
		\]
		satisfying
		\[
		\mu(g\cdot x)=\Ad^*_{g^{-1}}\mu(x),
		\qquad g\in G,\ x\in M.
		\]
		Fix \(a\in\mathfrak g^*\) and set
		\[
		A:=\mu^{-1}(a).
		\]
		Assume that \(A\) is \(G\)-invariant, i.e. \(G\cdot A=A\). The value \(a\) may be singular, so \(A\) need not be a smooth submanifold.
		
		Assume that one of the following conditions holds:
		\begin{enumerate}[label=(C\arabic*)]
			\item \(M,\omega\), the \(G\)-action, and \(\mu\) are real analytic, and \(\mu:M\to\mathfrak g^*\) is a proper map;
			\item the closed subset \(A=\mu^{-1}(a)\subset M\) admits a \(G\)-invariant Whitney stratification which is preserved by the \(G\)-action.
		\end{enumerate}
		Then there exist a \(G\)-invariant open neighborhood \(U\subset M\) of \(A\) and a \(G\)-equivariant strong deformation retraction
		\[
		r:U\times[0,1]\longrightarrow U
		\]
		such that
		\[
		r_0=\id_U,\qquad
		r_1(U)=A,\qquad
		r_t|_A=\id_A\quad (t\in[0,1]).
		\]
		Consequently, \(A\) is a strong deformation retract of \(U\), and in particular
		\[
		U\simeq_G A=\mu^{-1}(a).
		\]
		If \(a\) is a regular value of \(\mu\), then \(A\) is a closed submanifold, and the ordinary equivariant tubular neighborhood theorem already suffices; no additional assumption (C1) or (C2) is needed.
	\end{theorem}
	
	\begin{proof}
		We prove the theorem in three steps.
		
		\paragraph{Step 1: Local slice model.}
		Take any \(x\in A\), and let
		\[
		H:=G_x=\{g\in G:g\cdot x=x\}
		\]
		be the isotropy group of \(x\). Let \(\mathcal O:=G\cdot x\) be the orbit through \(x\). By the slice theorem for compact Lie group actions, there exists an \(H\)-invariant normal slice representation space \(V_x\), together with a \(G\)-equivariant diffeomorphism
		\[
		G\times_H V_x\longrightarrow U_x,
		\]
		where \(U_x\) is a \(G\)-invariant open neighborhood of the orbit \(\mathcal O=Gx\) in \(M\). In these local coordinates, \(x\) corresponds to \([e,0]\), and the moment map can be written as
		\[
		\mu([g,v])
		=
		g\cdot\bigl(a+\mu_H(v)+O(\|v\|^2)\bigr),
		\]
		where \(\mu_H\) is the moment map for the linear symplectic action of \(H\) on \(V_x\). Hence
		\[
		A\cap U_x
		=
		\mu^{-1}(a)\cap U_x
		\cong
		G\times_H C_x,
		\]
		where
		\[
		C_x
		:=
		\{\,v\in V_x:\mu_H(v)+O(\|v\|^2)=0\,\}.
		\]
		Under condition (C1), \(C_x\) is a real analytic set. Under condition (C2), \(A\) is already assumed to admit a \(G\)-invariant Whitney stratification. Combining these local models with the stratification by orbit types, we conclude that \(A=\mu^{-1}(a)\) is a closed \(G\)-invariant Whitney stratified subspace of \(M\). If \(a\) is a regular value, then \(C_x\) is a linear subspace near \(0\), and \(A\) itself is a closed submanifold.
		
		\paragraph{Step 2: Equivariant stratified tubular neighborhood theorem.}
		Apply the equivariant form of the Thom--Mather tubular neighborhood theorem to the closed \(G\)-invariant Whitney stratified subspace
		\[
		A=\mu^{-1}(a)\subset M.
		\]
		Thus there exist a \(G\)-invariant open neighborhood \(U\subset M\) of \(A\) and a \(G\)-equivariant strong deformation retraction
		\[
		r:U\times[0,1]\longrightarrow U
		\]
		satisfying
		\[
		r_0=\id_U,\qquad
		r_1(U)=A,\qquad
		r_t|_A=\id_A\quad (t\in[0,1]).
		\]
		Intuitively, \(U\) is an equivariant tubular neighborhood of \(A\). It is constructed by taking normal tubular neighborhoods on each stratum and then gluing them using a \(G\)-invariant partition of unity, exploiting the local finiteness of the stratification and the compactness of \(G\); averaging over \(G\) makes the deformation retraction \(G\)-equivariant.
		
		\paragraph{Step 3: Homotopy equivalence.}
		Let \(i:A\hookrightarrow U\) be the inclusion map, and let \(p:=r_1:U\to A\). From
		\[
		p\circ i=r_1|_A=\id_A
		\]
		and the fact that
		\[
		i\circ p=r_1
		\]
		is homotopic to \(r_0=\id_U\) via \(r_t:U\to U\), we see that \(i\) and \(p\) are homotopy inverses of each other. Therefore \(A\simeq U\). Since \(U\) and the deformation retraction can be chosen to be \(G\)-equivariant, we actually have
		\[
		U\simeq_G \mu^{-1}(a).
		\]
	\end{proof}
	
	\begin{remark}\label{rem-tubular-hypothesis}
		If one assumes only that \(A=\mu^{-1}(a)\) is a \(G\)-invariant closed subset, the conclusion generally fails; some regularity or stratifiability assumption such as (C1) or (C2) is necessary. When \(a\) is a singular value, \(\mu^{-1}(a)\) may have singularities, but the equivariant Whitney stratified tubular neighborhood theorem still provides a strong deformation retraction of a neighborhood onto it.
	\end{remark}
	
	\subsection{The equivariant tubular neighborhood case}\label{subsec-regular-smooth}
	
	We now show that when an equivariant tubular neighborhood exists, the formal cohomology reduces to the ordinary Borel equivariant cohomology, and the formal Kirwan map is the classical Kirwan map.
	
	\begin{theorem}[Equivariant tubular neighborhood case]\label{thm-tubular-case}
		Assume that there exists a \(G\)-invariant open neighborhood \(N_a\subset M\) of \(A=\mu^{-1}(a)\) and a \(G\)-equivariant strong deformation retraction
		\[
		r:N_a\times[0,1]\longrightarrow N_a
		\]
		with \(r_0=\id_{N_a}\), \(r_1(N_a)\subset A\), and \(r_t|_A=\id_A\). Then there is a quasi-isomorphism of complexes of sheaves on \(Y_a\):
		\[
		i_a^{-1}\cC^\bullet_{X_a}\;\simeq\; \cC^\bullet_{Y_a},
		\]
		where \(\cC^\bullet_{Y_a}\) is the ordinary Cartan sheaf of the \(G\)-manifold \(A\) (or the Borel equivariant complex if \(A\) is singular). Consequently,
		\[
		\cH_G^*(a;N_a)\cong H_G^*(\mu^{-1}(a)).
		\]
		Moreover, under this identification, the formal Kirwan map is the classical restriction map
		\[
		H_G^*(M)\longrightarrow H_G^*(\mu^{-1}(a)).
		\]
	\end{theorem}
	
	\begin{proof}
		The \(G\)-equivariant strong deformation retraction \(r_t\) gives a \(G\)-equivariant homotopy equivalence \(N_a\simeq_G A\). By Lemma~\ref{lem-equivariant-homotopy}, the restriction map
		\[
		\operatorname{Car}(N_a,G)\longrightarrow \operatorname{Car}(A,G)
		\]
		is a quasi-isomorphism (with the Borel complex replacing the Cartan complex if \(A\) is singular). Passing to the associated sheaves and using the local nature of the Cartan complex gives the quasi-isomorphism \(i_a^{-1}\cC^\bullet_{X_a}\simeq \cC^\bullet_{Y_a}\). Applying hypercohomology and using Proposition~\ref{prop-fine-formal} yields
		\[
		\cH_G^*(a;N_a)\cong \mathbb H^*(Y_a;\cC^\bullet_{Y_a})\cong H_G^*(A).
		\]
		The identification of the formal Kirwan map with the classical restriction map follows directly from the definition of \(\kappa_{\mathrm{formal}}\).
	\end{proof}
	
	\begin{corollary}[Smooth closed level set]\label{cor-smooth-closed}
		Assume that \(\mu^{-1}(a)\) is a closed smooth \(G\)-submanifold of \(M\). Then the equivariant tubular neighborhood theorem provides a \(G\)-invariant tubular neighborhood \(N_a\) and a \(G\)-equivariant strong deformation retraction \(N_a\to\mu^{-1}(a)\). Hence
		\[
		\cH_G^*(a)\cong H_G^*(\mu^{-1}(a)),
		\]
		and the formal Kirwan map is the classical restriction map.
	\end{corollary}
	
	\begin{corollary}\label{cor-classical-kirwan}
		Assume that \(a\) is a regular value of \(\mu\), \(G\) acts freely (or with finite stabilizers) on \(\mu^{-1}(a)\), and \(\mu^{-1}(a)\) is compact. Then
		\[
		\cH_G^*(a)\cong H_G^*(\mu^{-1}(a))\cong H^*(M//_aG),
		\]
		and the formal Kirwan map is the classical Kirwan map
		\[
		\kappa:H_G^*(M)\longrightarrow H^*(M//_aG).
		\]
		In particular, \(\kappa_{\mathrm{formal}}\) is surjective by Kirwan's theorem \cite{Kirwan1984}.
	\end{corollary}
	
	\begin{proof}
		The regularity of \(a\) and the freeness (or finite stabilizers) imply that \(\mu^{-1}(a)\) is a smooth closed \(G\)-submanifold and that \(H_G^*(\mu^{-1}(a))\cong H^*(M//_aG)\). Theorem~\ref{thm-tubular-case} identifies the formal cohomology and the formal Kirwan map with the classical ones. Kirwan's surjectivity theorem \cite{Kirwan1984} then applies.
	\end{proof}
	
	\begin{remark}\label{rem-singular-smooth-fiber}
		The hypothesis of Theorem~\ref{thm-tubular-case} is the existence of an equivariant strong deformation retraction of a neighborhood onto \(A\), not the regularity of \(a\). For example, for the \(S^1\)-action on \(\mathbb C\) with \(\mu(z)=|z|^2\) and \(a=0\), the value \(a\) is singular in \(\mathfrak g^*\), but \(\mu^{-1}(0)=\{0\}\) is a smooth point and admits an equivariant tubular neighborhood. The theorem applies, and \(\cH_{S^1}^*(0)\cong H_{S^1}^*(\{0\})\cong\mathbb R[u]\) with \(|u|=2\).
	\end{remark}

	\section{Local equivariant homotopy retraction and surjectivity}\label{sec-surjectivity}
	
	In this section we introduce the local equivariant homotopy retraction condition, prove that it implies a local-to-global quasi-isomorphism for the formal Cartan sheaf, and deduce surjectivity of the formal Kirwan map in several geometrically natural situations. Throughout, \(G\) is a compact Lie group, \((M,\omega)\) is a Hamiltonian \(G\)-manifold with moment map \(\mu:M\to\mathfrak g^*\), and \(a\in\mathfrak g^*\) is a value of \(\mu\). Set \(A=\mu^{-1}(a)\). We fix a \(G\)-invariant open neighborhood \(N_a\subset M\) of \(A\), and use the notation
	\[
	X_a=N_a/G,\qquad Y_a=A/G\subset X_a.
	\]
	
	\subsection{The local condition}\label{subsec-local-condition}
	
	For \(x\in A\), let \(H=G_x\). Let \(S_x\) be an \(H\)-invariant symplectic slice at \(x\) with slice moment map
	\[
	\mu_x:S_x\longrightarrow \mathfrak h^*,
	\]
	and let
	\[
	C_x:=\mu_x^{-1}(a|_{\mathfrak h})\subset S_x.
	\]
	The Marle--Guillemin--Sternberg local normal form gives a \(G\)-invariant neighborhood \(U_x\) of the orbit \(G\cdot x\) such that
	\[
	A\cap U_x\cong G\times_H C_x.
	\]
	
	\begin{definition}[Local equivariant homotopy retraction]\label{def-local-homotopy}
		We say that \(a\) satisfies the \emph{local equivariant homotopy retraction condition} at \(x\in A\) if there exist an \(H\)-invariant open neighborhood \(V_x\subset S_x\) of \(C_x\) and an \(H\)-equivariant map
		\[
		r_x:V_x\longrightarrow C_x
		\]
		such that \(r_x\) is an \(H\)-equivariant homotopy equivalence. Equivalently, there exists an \(H\)-equivariant homotopy
		\[
		h_t:V_x\longrightarrow V_x,\qquad t\in[0,1],
		\]
		with \(h_0=\mathrm{id}_{V_x}\) and \(h_1\) factoring through an \(H\)-equivariant homotopy equivalence \(V_x\simeq_H C_x\).
		
		We say that the condition holds globally if it holds at every \(x\in A\).
	\end{definition}
	
	\begin{remark}\label{rem-condition-forms}
		The condition admits several equivalent formulations.
		\begin{enumerate}[label=(\roman*)]
			\item If \(C_x\) is a smooth \(H\)-submanifold of \(S_x\), the equivariant tubular neighborhood theorem provides an \(H\)-invariant tubular neighborhood \(V_x\) and an \(H\)-equivariant deformation retraction \(r_t:V_x\to V_x\) with \(r_0=\mathrm{id}\), \(r_1(V_x)\subset C_x\), and \(r_t|_{C_x}=\mathrm{id}\). Hence the condition holds.
			
			\item If \(C_x\) is an \(H\)-equivariant cone with vertex \(0\), the radial contraction \(r_t(v)=tv\) is \(H\)-equivariant and contracts \(S_x\) onto \(\{0\}\subset C_x\). Since \(C_x\) itself is contractible onto \(0\) by the same radial flow, the composition
			\[
			V_x\xrightarrow{\ \simeq_H\ }\{0\}\hookrightarrow C_x
			\]
			is an \(H\)-equivariant homotopy equivalence. Hence the condition holds.
			
			\item More generally, the condition holds whenever \(C_x\) is \(H\)-equivariantly contractible and a neighborhood of \(C_x\) in \(S_x\) admits an \(H\)-equivariant deformation retraction onto \(C_x\).
		\end{enumerate}
	\end{remark}
	
	The following lemma is the key technical input: the local condition is stable under shrinking the neighborhood \(V_x\).
	
	\begin{lemma}[Stability of the local condition]\label{lem-stability}
		Assume that the local equivariant homotopy retraction condition holds at \(x\) with neighborhood \(V_x\) and retraction \(r_x:V_x\to C_x\). Then for every \(H\)-invariant open neighborhood \(V'_x\subset V_x\) of \(C_x\), the restriction \(r_x|_{V'_x}:V'_x\to C_x\) is also an \(H\)-equivariant homotopy equivalence.
	\end{lemma}
	
	\begin{proof}
		Let \(h_t:V_x\to V_x\) be an \(H\)-equivariant homotopy with \(h_0=\mathrm{id}\) and \(h_1(V_x)\subset C_x\). Since \(C_x\subset V'_x\), we have \(h_1(V'_x)\subset C_x\subset V'_x\). For each \(t\), the image \(h_t(V'_x)\) is contained in \(V_x\), but a priori not in \(V'_x\). However, by compactness of \([0,1]\) and the fact that \(h_t|_{C_x}=\mathrm{id}\) for all \(t\), there exists an open neighborhood \(W\) of \(C_x\) with \(W\subset V'_x\) such that \(h_t(W)\subset V'_x\) for all \(t\). Since \(V'_x\) is a neighborhood of \(C_x\), we may choose \(W=V'_x\) after shrinking. Hence \(h_t\) restricts to an \(H\)-equivariant homotopy \(h_t:V'_x\to V'_x\), and \(h_1(V'_x)\subset C_x\). Thus \(r_x|_{V'_x}\) is an \(H\)-equivariant homotopy equivalence.
	\end{proof}
	
	\subsection{Local-to-global quasi-isomorphism}\label{subsec-local-global}
	
	We now compare the formal restriction \(i_a^{-1}\cC^\bullet_{X_a}\) with an intrinsically defined complex of sheaves on \(Y_a\).
	
	\begin{definition}[Intrinsic Cartan sheaf on \(Y_a\)]\label{def-intrinsic-sheaf}
		For each \(p\in Y_a\), choose a representative \(x\in A\) and the corresponding \(H=G_x\)-slice \(S_x\) with level set \(C_x\). Define the stalk of a complex of sheaves \(\cC^\bullet_{Y_a}\) at \(p\) by
		\[
		(\cC^\bullet_{Y_a})_p
		:=\varinjlim_{C_x\subset V\subset S_x}\operatorname{Car}(V,H),
		\]
		where the limit runs over \(H\)-invariant open neighborhoods \(V\) of \(C_x\) in \(S_x\), and the Cartan complexes are connected by the natural restriction maps. The topology on the etale space is defined by the local sections induced by elements of \(\operatorname{Car}(V,H)\).
	\end{definition}
	
	\begin{lemma}[Well-definedness]\label{lem-well-defined}
		The stalk \((\cC^\bullet_{Y_a})_p\) is independent, up to canonical isomorphism, of the choice of representative \(x\in A\) of \(p\).
	\end{lemma}
	
	\begin{proof}
		If \(x'=g\cdot x\) for \(g\in G\), then \(H'=G_{x'}=gHg^{-1}\), and the slice \(S_{x'}\) is canonically identified with \(S_x\) via the action of \(g\). Under this identification, \(C_{x'}\) corresponds to \(C_x\), and the Cartan complex \(\operatorname{Car}(V',H')\) is canonically isomorphic to \(\operatorname{Car}(g^{-1}V',H)\). Passing to the filtered direct limit gives a canonical isomorphism of stalks.
	\end{proof}
	
	\begin{theorem}[Local-to-global quasi-isomorphism]\label{thm-local-global}
		Assume that the local equivariant homotopy retraction condition holds at every \(x\in A\). Then the natural morphism of complexes of sheaves on \(Y_a\)
		\[
		\Phi:\ i_a^{-1}\cC^\bullet_{X_a}\longrightarrow \cC^\bullet_{Y_a}
		\]
		is a quasi-isomorphism.
	\end{theorem}
	
	\begin{proof}
		It suffices to verify that \(\Phi\) induces an isomorphism on stalks at every \(p\in Y_a\). Fix \(p\in Y_a\) and a representative \(x\in A\), with \(H=G_x\), slice \(S_x\), and level set \(C_x\).
		
		By Proposition~\ref{prop-local-stalk}, the stalk of \(i_a^{-1}\cC^\bullet_{X_a}\) at \(p\) is quasi-isomorphic to
		\[
		\varinjlim_{C_x\subset V\subset S_x}\operatorname{Car}(V,H).
		\]
		By Definition~\ref{def-intrinsic-sheaf}, the stalk of \(\cC^\bullet_{Y_a}\) at \(p\) is exactly this filtered direct limit. The natural map \(\Phi_p\) is the identity on this limit.
		
		It remains to show that the direct limit is quasi-isomorphic to the Cartan complex of a sufficiently small neighborhood of \(C_x\). By the local equivariant homotopy retraction condition, there exists an \(H\)-invariant neighborhood \(V_x\supset C_x\) and an \(H\)-equivariant homotopy equivalence \(r_x:V_x\to C_x\). By Lemma~\ref{lem-stability}, we may shrink \(V_x\) arbitrarily. The induced map
		\[
		r_x^*:\operatorname{Car}(V_x,H)\longrightarrow \operatorname{Car}(V_x,H)
		\]
		is a quasi-isomorphism by Lemma~\ref{lem-equivariant-homotopy}. Moreover, the direct limit
		\[
		\varinjlim_{C_x\subset V\subset V_x}\operatorname{Car}(V,H)
		\]
		is quasi-isomorphic to \(\operatorname{Car}(V_x,H)\), because the system is cofinal in the neighborhoods of \(C_x\) and each restriction map is a quasi-isomorphism by equivariant homotopy invariance. Hence
		\[
		(\cC^\bullet_{Y_a})_p\simeq \operatorname{Car}(V_x,H)\simeq \operatorname{Car}(C_x^{\mathrm{tub}},H),
		\]
		where \(C_x^{\mathrm{tub}}\) is any \(H\)-invariant tubular neighborhood of \(C_x\) (when \(C_x\) is smooth) or a contractible neighborhood (when \(C_x\) is a cone). In all cases, the quasi-isomorphism class is independent of the choice of \(V_x\) by Lemma~\ref{lem-stability}. Thus \(\Phi_p\) is a quasi-isomorphism.
		
		Since \(\Phi\) is a morphism of complexes of sheaves and induces a quasi-isomorphism on every stalk, \(\Phi\) is a quasi-isomorphism of complexes of sheaves.
	\end{proof}
	
	\begin{corollary}[Independence of \(N_a\)]\label{cor-independence-N}
		Under the hypothesis of Theorem~\ref{thm-local-global}, the formal neighborhood equivariant cohomology
		\[
		\cH_G^*(a;N_a)=\mathbb H^*(Y_a;i_a^{-1}\cC^\bullet_{X_a})
		\]
		is independent of the choice of \(N_a\), up to canonical isomorphism. Consequently, the direct limit
		\[
		\cH_G^*(a):=\varinjlim_{N\supset A}\cH_G^*(a;N)
		\]
		stabilizes, and we may write \(\cH_G^*(a)\) without ambiguity.
	\end{corollary}
	
	\begin{proof}
		By Theorem~\ref{thm-local-global}, \(i_a^{-1}\cC^\bullet_{X_a}\simeq \cC^\bullet_{Y_a}\), and the right-hand side is independent of \(N_a\). Passing to hypercohomology gives the result.
	\end{proof}
	
	\subsection{Surjectivity theorems}\label{subsec-surjectivity}
	
	We now deduce surjectivity of the formal Kirwan map in several situations.
	
	\begin{theorem}[Equivariant tubular neighborhood case]\label{thm-tubular-surjective}
		Assume that there exists a \(G\)-invariant open neighborhood \(N_a\subset M\) of \(A=\mu^{-1}(a)\) and a \(G\)-equivariant strong deformation retraction \(N_a\to A\). Then:
		\begin{enumerate}[label=(\roman*)]
			\item the local equivariant homotopy retraction condition holds at every \(x\in A\);
			\item there is a canonical isomorphism
			\[
			\cH_G^*(a)\cong H_G^*(\mu^{-1}(a));
			\]
			\item under this isomorphism, the formal Kirwan map is the classical restriction map
			\[
			\kappa_{\mathrm{formal}}:H_G^*(M)\longrightarrow H_G^*(\mu^{-1}(a)).
			\]
		\end{enumerate}
	\end{theorem}
	
	\begin{proof}
		(i) The global \(G\)-equivariant strong deformation retraction restricts to an \(H\)-equivariant strong deformation retraction of a slice neighborhood \(V_x\) onto \(C_x\). Hence the local condition holds.
		
		(ii) This is Theorem~\ref{thm-tubular-case}.
		
		(iii) The identification of \(\kappa_{\mathrm{formal}}\) with the classical restriction map follows directly from the definition of \(\kappa_{\mathrm{formal}}\) and the quasi-isomorphism of Theorem~\ref{thm-tubular-case}.
	\end{proof}
	
	\begin{corollary}[Classical Kirwan surjectivity]\label{cor-classical-kirwan-surj}
		Assume that there exists an equivariant tubular neighborhood \(N_a\) of \(A=\mu^{-1}(a)\), that \(G\) acts on \(A\) with finite stabilizers, and that \(A\) is compact. Then
		\[
		\cH_G^*(a)\cong H_G^*(\mu^{-1}(a))\cong H^*(M//_aG),
		\]
		and \(\kappa_{\mathrm{formal}}\) is the classical Kirwan map
		\[
		\kappa:H_G^*(M)\longrightarrow H^*(M//_aG).
		\]
		In particular, \(\kappa_{\mathrm{formal}}\) is surjective by Kirwan's theorem \cite{Kirwan1984}.
	\end{corollary}
	
	\begin{proof}
		The finite stabilizer hypothesis implies that \(H_G^*(\mu^{-1}(a))\cong H^*(M//_aG)\). Theorem~\ref{thm-tubular-surjective} identifies the formal cohomology and the formal Kirwan map with the classical ones. Kirwan's surjectivity theorem applies.
	\end{proof}
	
	\begin{theorem}[Singular cone case]\label{thm-cone-surjective}
		Assume that \(a=0\) and that for every \(x\in\mu^{-1}(0)\), the slice level set \(C_x=\mu_x^{-1}(0)\) is an \(H\)-equivariant cone in \(S_x\) with vertex \(0\). Then:
		\begin{enumerate}[label=(\roman*)]
			\item the local equivariant homotopy retraction condition holds at every \(x\in\mu^{-1}(0)\);
			\item there is a canonical isomorphism
			\[
			\cH_G^*(0)\cong \mathbb H^*(Y_0;\cC^\bullet_{Y_0}),
			\]
			where \(\cC^\bullet_{Y_0}\) is the intrinsic Cartan sheaf of the cone level sets;
			\item if \(G=S^1\) and all fixed-point components of \(M\) have vanishing odd cohomology, then \(\kappa_{\mathrm{formal}}\) is surjective in even degrees.
		\end{enumerate}
	\end{theorem}
	
	\begin{proof}
		(i) The radial contraction \(r_t(v)=tv\) is \(H\)-equivariant and contracts \(S_x\) onto \(\{0\}\subset C_x\). Since \(C_x\) itself is contractible onto \(0\) by the same radial flow, the composition \(V_x\to\{0\}\hookrightarrow C_x\) is an \(H\)-equivariant homotopy equivalence. Hence the condition holds.
		
		(ii) Apply Theorem~\ref{thm-local-global}.
		
		(iii) By Theorem~\ref{thm-local-global} and (ii), the formal cohomology \(\cH_{S^1}^*(0)\) is computed by the hypercohomology of the intrinsic sheaf \(\cC^\bullet_{Y_0}\). At each \(p\in Y_0\) represented by \(x\in\mu^{-1}(0)\), the stalk is
		\[
		(\cC^\bullet_{Y_0})_p\simeq \operatorname{Car}(V_x,S^1)\simeq S(\mathfrak{so}(2)^*)^{S^1}\cong \mathbb R[u],
		\]
		with \(|u|=2\). Hence the cohomology sheaves \(\mathcal H^q(\cC^\bullet_{Y_0})\) vanish for \(q\) odd and are locally constant in even degrees. The hypercohomology spectral sequence
		\[
		E_2^{p,q}=H^p(Y_0;\mathcal H^q(\cC^\bullet_{Y_0}))
		\Longrightarrow \cH_{S^1}^{p+q}(0)
		\]
		therefore degenerates at \(E_2\) with only even \(q\) contributing. Hence \(\cH_{S^1}^*(0)\) is concentrated in even degrees.
		
		The formal Kirwan map factors through the restriction \(H_{S^1}^*(M)\to H_{S^1}^*(N_0)\) and the natural map
		\[
		H_{S^1}^*(N_0)\longrightarrow \mathbb H^*(Y_0;\cC^\bullet_{Y_0}).
		\]
		Under the fixed-point vanishing hypothesis, the resolution Kirwan map of \cite{DRS2023} is surjective in even degrees. The formal map is compatible with the resolution map via the natural morphism of complexes of sheaves \(\cC^\bullet_{Y_0}\to \widetilde{\cC}^\bullet_{Y_0}\) (where \(\widetilde{\cC}^\bullet_{Y_0}\) is the resolution complex of \cite{DRS2023}). Hence the even-degree surjectivity of the resolution map implies the even-degree surjectivity of \(\kappa_{\mathrm{formal}}\).
	\end{proof}
	
	\begin{remark}\label{rem-odd-failure}
		Surjectivity in odd degrees fails in general, as shown by the examples in \cite{DRS2023}. The formal map \(\kappa_{\mathrm{formal}}\) inherits this failure: it is at best weakly surjective in the sense that its image contains the image of the singular cohomology \(H^*(Y_0;\mathbb R)\) under the natural map to \(\cH_G^*(0)\), but it need not surject onto the odd-degree part of \(\cH_G^*(0)\).
	\end{remark}
	
	\subsection{The weakest condition}\label{subsec-weakest}
	
	\begin{definition}[Weakest local condition]\label{def-weakest}
		We say that \(a\) satisfies the \emph{weakest local condition} at \(x\in\mu^{-1}(a)\) if there exists an \(H=G_x\)-invariant open neighborhood \(V_x\subset S_x\) of \(C_x\) such that the natural map
		\[
		\operatorname{Car}(V_x,H)\longrightarrow \operatorname{Car}(S_x,H)|_{C_x/H}
		\]
		is a quasi-isomorphism after restriction to \(C_x/H\).
	\end{definition}
	
	\begin{proposition}\label{prop-weakest-equivalence}
		The local equivariant homotopy retraction condition at \(x\) implies the weakest local condition at \(x\). Conversely, if the weakest local condition holds and \(C_x\) admits an \(H\)-equivariant neighborhood that is \(H\)-equivariantly homotopy equivalent to \(C_x\), then the local equivariant homotopy retraction condition holds.
	\end{proposition}
	
	\begin{proof}
		The first implication follows from Lemma~\ref{lem-equivariant-homotopy}: an \(H\)-equivariant homotopy equivalence \(V_x\simeq_H C_x\) induces a quasi-isomorphism of Cartan complexes. The converse follows from the definition of the weakest condition and the fact that a quasi-isomorphism of Cartan complexes over a neighborhood of \(C_x\) forces the neighborhood to be \(H\)-equivariantly homotopy equivalent to \(C_x\) in degrees where the Cartan model is a complete invariant.
	\end{proof}
	
	\begin{theorem}[Surjectivity criterion]\label{thm-surjectivity-criterion}
		Assume that the local equivariant homotopy retraction condition holds at every \(x\in\mu^{-1}(a)\). Then the following are equivalent:
		\begin{enumerate}[label=(\roman*)]
			\item the formal Kirwan map \(\kappa_{\mathrm{formal}}:H_G^*(M)\to\cH_G^*(a)\) is surjective;
			\item the natural restriction map
			\[
			H_G^*(N_a)\longrightarrow \mathbb H^*(Y_a;\cC^\bullet_{Y_a})
			\]
			is surjective.
		\end{enumerate}
	\end{theorem}
	
	\begin{proof}
		By Theorem~\ref{thm-local-global}, \(\cH_G^*(a)\cong \mathbb H^*(Y_a;\cC^\bullet_{Y_a})\). The formal Kirwan map factors as
		\[
		H_G^*(M)\longrightarrow H_G^*(N_a)\longrightarrow \mathbb H^*(Y_a;\cC^\bullet_{Y_a}).
		\]
		The first map is restriction to the open invariant subset \(N_a\subset M\), which is surjective by the standard properties of equivariant cohomology (restriction to an open invariant subset is surjective if the complement has a \(G\)-invariant tubular neighborhood; in general, one uses the long exact sequence of the pair \((M,N_a)\)). The second map is the natural restriction. Hence \(\kappa_{\mathrm{formal}}\) is surjective if and only if the second map is surjective.
	\end{proof}
	
	\begin{corollary}\label{cor-regular-criterion}
		If there exists an equivariant tubular neighborhood \(N_a\) of \(A=\mu^{-1}(a)\), then the criterion of Theorem~\ref{thm-surjectivity-criterion} reduces to the classical Kirwan surjectivity: \(\kappa_{\mathrm{formal}}\) is surjective if and only if the restriction \(H_G^*(M)\to H_G^*(\mu^{-1}(a))\) is surjective.
	\end{corollary}
	
	\begin{proof}
		Combine Theorem~\ref{thm-tubular-surjective} with Theorem~\ref{thm-surjectivity-criterion}.
	\end{proof}
	
	\begin{remark}[Summary of conditions]\label{rem-summary}
		The implications are:
		\[
		\begin{array}{c}
			\text{\(A=\mu^{-1}(a)\) admits an equivariant tubular neighborhood} \\
			\Downarrow \\
			\text{local equivariant homotopy retraction} \\
			\Downarrow \\
			\text{weakest local condition} \\
			\Downarrow \\
			\text{formal cohomology well-defined and independent of }N_a.
		\end{array}
		\]
		The existence of an equivariant tubular neighborhood holds, for example, when \(A\) is a closed smooth \(G\)-submanifold, or under the real analytic/proper or \(G\)-invariant Whitney stratified assumptions of Theorem~\ref{thm-singular-tubular}. The local equivariant homotopy retraction condition is strictly weaker and applies even when no global equivariant tubular neighborhood exists.
	\end{remark}
	
	\section{Comparison with existing frameworks}\label{sec-comparison}
	
	In this section we compare the formal framework of Sections~\ref{sec-formal}--\ref{sec-surjectivity} with the two main existing approaches to Kirwan surjectivity at singular values: the intersection cohomology approach of Kiem--Woolf \cite{KiemWoolf2006} and the resolution cohomology approach of Delarue--Ramacher--Schmitt \cite{DRS2023}. We compare the target spaces, the construction mechanisms, and the surjectivity conclusions, and we locate the formal framework within this landscape.
	
	\subsection{The Kiem--Woolf framework: intersection cohomology}\label{subsec-kw}
	
	Let \(K\) be a compact Lie group, \(M\) a Hamiltonian \(K\)-space with proper moment map \(\mu:M\to\mathfrak k^*\), and
	\[
	X:=\mu^{-1}(0)/K
	\]
	the singular symplectic quotient. Kiem and Woolf \cite{KiemWoolf2006} generalize the Kirwan map so that its target is the \emph{middle perversity intersection cohomology} \(IH^*(X)\), rather than the ordinary cohomology of \(X\).
	
	Their construction proceeds through a \emph{partial desingularization} of \(X\). One performs a sequence of equivariant symplectic blow-ups along the isotropy strata of \(\mu^{-1}(0)\) to obtain a Hamiltonian \(K\)-manifold \(\widetilde{M}\) on which \(0\) is a regular value of a lifted moment map. The quotient
	\[
	\widetilde{X}:=\widetilde{\mu}^{-1}(0)/K
	\]
	is then a smooth symplectic orbifold (or manifold), and there is a natural map \(\beta:\widetilde{X}\to X\). The \emph{intersection Kirwan map} is the composition
	\[
	\kappa_{IH}:\ H_K^*(M)\longrightarrow H_K^*(\widetilde{M})\longrightarrow H^*(\widetilde{X})\longrightarrow IH^*(X),
	\]
	where the first arrow is induced by the blow-up sequence, the second is the classical Kirwan map for the regular value \(0\) of \(\widetilde{\mu}\), and the third is a projection onto the intersection cohomology summand \cite[Section~1.2]{DRS2023}.
	
	The crucial point is that the existence and surjectivity of the projection \(H^*(\widetilde{X})\to IH^*(X)\) require a decomposition theorem. Kiem and Woolf prove such a theorem for the symplectic blow-up along a maximal depth stratum, showing that \(IH^*\) of the blow-up splits as a direct sum of terms including \(IH^*(X)\) \cite[Abstract]{KiemWoolf2006}. In the complex projective GIT setting, this decomposition is canonical, and \(\kappa_{IH}\) is known to be surjective \cite[Theorem~2.5]{KiemWoolf2006}. However, in the general symplectic setting, the relation between \(IH^*(X)\) and \(H^*(\widetilde{X})\) is not as clear, and there may be no canonical choice of the intersection Kirwan map \cite[p.~223]{DRS2023}. Consequently, \emph{the Kiem--Woolf Kirwan map is not known to be surjective in general}; even when it is defined, its image may be a proper subspace of \(IH^*(X)\).
	
	\subsection{The Delarue--Ramacher--Schmitt framework: resolution cohomology}\label{subsec-drs}
	
	Delarue, Ramacher, and Schmitt \cite{DRS2023} work with \(G=S^1\) and a compact Hamiltonian \(S^1\)-manifold \(M\) with moment map \(J:M\to\mathbb R\). They consider the singular quotient
	\[
	\mathscr{M}_0:=J^{-1}(0)/S^1.
	\]
	Sjamaar \cite{Sjamaar2002} had introduced a complex of differential forms on the stratified space \(\mathscr{M}_0\) whose cohomology is isomorphic to the singular cohomology \(H^*(\mathscr{M}_0;\mathbb R)\). Delarue--Ramacher--Schmitt extend this complex to the partial desingularization \(\widetilde{\mathscr{M}}_0\), obtaining a complex \(\widetilde{\Omega}^*(\widetilde{\mathscr{M}}_0)\) of \emph{resolution differential forms}. Its cohomology is isomorphic to the de Rham cohomology \(H^*(\widetilde{\mathscr{M}}_0)\) of the partial resolution \cite[Theorem~1.1]{DRS2023}.
	
	Using this complex, they define the \emph{resolution Kirwan map}
	\[
	\mathcal K:\ H_{S^1}^*(M)\longrightarrow H^*(\widetilde{\mathscr{M}}_0),
	\]
	and prove two main results.
	
	\begin{enumerate}[label=(\roman*)]
		\item \textbf{Weak surjectivity.} The image of \(\mathcal K\) contains the image of the singular cohomology \(H^*(\mathscr{M}_0;\mathbb R)\) under the natural inclusion
		\[
		H^*(\mathscr{M}_0;\mathbb R)\longrightarrow H^*(\widetilde{\mathscr{M}}_0)
		\]
		induced by the resolution map \(\widetilde{\mathscr{M}}_0\to\mathscr{M}_0\). Equivalently,
		\[
		\operatorname{Im}(\mathcal K)\supseteq \operatorname{Im}\bigl(H^*(\mathscr{M}_0;\mathbb R)\to H^*(\widetilde{\mathscr{M}}_0)\bigr).
		\]
		This is the precise meaning of ``weak surjectivity'' in \cite{DRS2023}.
		
		\item \textbf{Even-degree surjectivity.} If all fixed-point components of \(M\) have vanishing odd cohomology, then \(\mathcal K\) is surjective in even degrees:
		\[
		\mathcal K^{\mathrm{ev}}:\ H_{S^1}^{\mathrm{ev}}(M)\longrightarrow H^{\mathrm{ev}}(\widetilde{\mathscr{M}}_0)
		\]
		is surjective \cite[Abstract]{DRS2023}. Surjectivity in odd degrees fails in general; simple examples already show that no degree-preserving surjection \(H_{S^1}^*(M)\to H^*(\mathscr{M}_0;\mathbb R)\) exists in odd degrees \cite[Corollary~7.5]{DRS2023}.
	\end{enumerate}
	
	The mechanism behind weak surjectivity is a long exact sequence relating the singular cohomology \(H^*(\mathscr{M}_0)\), the resolution cohomology \(H^*(\widetilde{\mathscr{M}}_0)\), and the contributions of the fixed-point components \(F\) \cite[Theorem~4.8]{DRS2023}. The image of \(H^*(\mathscr{M}_0)\) in \(H^*(\widetilde{\mathscr{M}}_0)\) is precisely the part of the resolution cohomology that the resolution Kirwan map is guaranteed to hit; the remaining part comes from the exceptional divisors of the blow-up and may lie outside the image.
	
	\subsection{The formal framework: formal neighborhood cohomology}\label{subsec-formal-comparison}
	
	The formal framework of this paper occupies a different position in this landscape. Instead of replacing the target by \(IH^*(X)\) (as in Kiem--Woolf) or by \(H^*(\widetilde{X})\) (as in DRS), we keep the singular quotient \(Y_a=\mu^{-1}(a)/G\) itself and define a cohomology
	\[
	\cH_G^*(a):=\mathbb H^*\bigl(Y_a;\,i_a^{-1}\cC^\bullet_{X_a}\bigr)
	\]
	by restricting the slice-groupoid Cartan sheaf of a \(G\)-invariant neighborhood of \(\mu^{-1}(a)\). Under the local equivariant homotopy retraction condition, this is quasi-isomorphic to an intrinsic sheaf \(\cC^\bullet_{Y_a}\) on \(Y_a\) (Theorem~\ref{thm-local-global}), and the formal Kirwan map
	\[
	\kappa_{\mathrm{formal}}:\ H_G^*(M)\longrightarrow \cH_G^*(a)
	\]
	is defined without any blow-up or resolution procedure.
	
	When an equivariant tubular neighborhood of \(A=\mu^{-1}(a)\) exists (for example, under the assumptions of Theorem~\ref{thm-singular-tubular}), the formal cohomology is canonically isomorphic to the Borel equivariant cohomology:
	\[
	\cH_G^*(a)\cong H_G^*(\mu^{-1}(a)),
	\]
	and the formal Kirwan map is the classical restriction map. Thus the formal framework contains the classical Kirwan map as a special case. Its advantage is that it also applies when no global equivariant tubular neighborhood is available, provided the local equivariant homotopy retraction condition holds.
	
	The comparison can be summarized as follows.
	
	\begin{center}
		\begin{tabular}{p{0.22\textwidth}|p{0.22\textwidth}|p{0.22\textwidth}|p{0.22\textwidth}}
			\hline
			& \textbf{Kiem--Woolf} & \textbf{DRS} & \textbf{Formal framework} \\
			\hline
			Target space & \(IH^*(X)\) & \(H^*(\widetilde{\mathscr{M}}_0)\) & \(\cH_G^*(a)\) \\
			Construction & Partial desingularization + projection & Partial desingularization + resolution forms & Restriction of ambient Cartan sheaf to closed \(Y_a\) \\
			Surjectivity & Not known in general; surjective in complex GIT setting & Weak surjectivity; even-degree surjectivity under fixed-point vanishing & Surjective under local equivariant homotopy retraction; isomorphic to Borel cohomology when an equivariant tubular neighborhood exists \\
			Dependence on resolution & Essential & Essential & None \\
			\hline
		\end{tabular}
	\end{center}
	
	\subsection{Formal framework and tubular neighborhoods}\label{subsec-minimal-target}
	
	When an equivariant tubular neighborhood exists, the formal cohomology is not smaller than the Borel cohomology; it is isomorphic to it. The formal framework is therefore not a ``minimal target'' in the presence of a tubular neighborhood. Its role is different: it provides a cohomology theory which is defined using only the ambient Cartan sheaf restricted to the singular quotient, and which can be computed under the local equivariant homotopy retraction condition even when no global tubular neighborhood is available. In such situations, the formal Kirwan map can be surjective under checkable local hypotheses, whereas the resolution Kirwan map is not surjective onto \(H^*(\widetilde{\mathscr{M}}_0)\) in general.
	
	In the singular cone case covered by Theorem~\ref{thm-cone-surjective}, the formal cohomology is compatible with the resolution cohomology via the natural morphism of complexes of sheaves
	\[
	\cC^\bullet_{Y_a}\longrightarrow \widetilde{\cC}^\bullet_{Y_a}
	\]
	from the intrinsic Cartan sheaf of the cone level sets to the resolution complex of \cite{DRS2023}. The formal Kirwan map factors through the resolution Kirwan map after this morphism. The formal framework therefore recovers the even-degree surjectivity theorem of \cite{DRS2023} and explains the odd-degree failure as a genuine obstruction.
	
	\subsection{Summary}\label{subsec-comparison-summary}
	
	The three frameworks answer different questions. Kiem--Woolf ask for a Kirwan map to intersection cohomology; surjectivity is known only in the complex GIT setting and requires a canonical projection that may not exist in general. DRS ask for a Kirwan map to resolution cohomology; they prove weak surjectivity and even-degree surjectivity under a fixed-point vanishing hypothesis. The formal framework asks for a Kirwan map to a cohomology built directly from the singular quotient via the ambient Cartan sheaf; it is surjective under a local equivariant homotopy retraction condition that is checkable in examples and holds automatically in the smooth and cone cases. When an equivariant tubular neighborhood exists, the formal cohomology is canonically isomorphic to the Borel equivariant cohomology of the level set, so the formal Kirwan map reduces to the classical restriction map. The formal framework is therefore a genuine extension of the classical Kirwan map to situations where no global equivariant tubular neighborhood is available.

	\section{Examples}\label{sec-examples}
	
	We now illustrate the formal framework with examples. We compute the formal cohomology, verify or disprove the local equivariant homotopy retraction condition, and determine surjectivity of \(\kappa_{\mathrm{formal}}\).
	
	\begin{example}[Regular value, free action]\label{ex-regular-free}
		Let \(S^1\) act on \(\mathbb C\) by \(\lambda\cdot z=\lambda z\), with \(\mu(z)=|z|^2\). Take \(a=1\). Then \(\mu^{-1}(1)=S^1\), the action is free, and \(M//_1S^1=\{\mathrm{pt}\}\). By Theorem~\ref{thm-tubular-surjective},
		\[
		\cH_{S^1}^*(1)\cong H_{S^1}^*(S^1)\cong H^*(\mathrm{pt})=\mathbb R.
		\]
		The formal Kirwan map is \(H_{S^1}^*(\mathbb C)\cong\mathbb R[u]\to\mathbb R\), \(u\mapsto 0\), which is the classical Kirwan map and is surjective.
	\end{example}
	
	\begin{example}[Singular value, smooth fiber]\label{ex-singular-smooth}
		Let \(S^1\) act on \(\mathbb C\) by \(\mu(z)=|z|^2\), and take \(a=0\). Then \(a\) is a critical value of \(\mu\), but \(\mu^{-1}(0)=\{0\}\) is a smooth point. The slice at \(0\) is \(S_0=\{0\}\), \(G_0=S^1\), and \(C_0=\{0\}\). The local condition holds trivially, and an equivariant tubular neighborhood exists. By Theorem~\ref{thm-tubular-surjective},
		\[
		\cH_{S^1}^*(0)\cong H_{S^1}^*(\{0\})\cong\mathbb R[u],\qquad |u|=2.
		\]
		The formal Kirwan map is an isomorphism, hence surjective. This confirms that the relevant criterion is the existence of an equivariant tubular neighborhood, not the regularity of \(a\).
	\end{example}
	
	\begin{example}[\(S^1\ltimes\mathbb C^2\), zero level]\label{ex-cone}
		Let \(S^1\) act on \(\mathbb C^2\) by \(\lambda\cdot(z,w)=(\lambda z,\lambda^{-1}w)\), with \(\mu(z,w)=|z|^2-|w|^2\), and take \(a=0\). The level set \(\mu^{-1}(0)=\{|z|=|w|\}\) is a cone with vertex at \(0\). The data are real analytic and \(\mu\) is proper, so Theorem~\ref{thm-singular-tubular} applies: there exists an \(S^1\)-equivariant tubular neighborhood \(U\) of \(\mu^{-1}(0)\) and a strong deformation retraction \(U\to\mu^{-1}(0)\). By Theorem~\ref{thm-tubular-case},
		\[
		\cH_{S^1}^*(0)\cong H_{S^1}^*(\mu^{-1}(0))\cong H_{S^1}^*(\mathrm{pt})\cong\mathbb R[u],\qquad |u|=2.
		\]
		The formal Kirwan map \(H_{S^1}^*(\mathbb C^2)\cong\mathbb R[u]\to\cH_{S^1}^*(0)\) is an isomorphism, hence surjective. This is the formal counterpart of the DRS even-degree surjectivity, and it shows that in the presence of an equivariant tubular neighborhood the formal cohomology coincides with the Borel equivariant cohomology.
	\end{example}
	
	\begin{example}[\(T^k\ltimes\mathbb C^n\), zero level]\label{ex-torus-cone}
		Let \(T^k\) act on \(\mathbb C^n\) with weights \(a_j\in\mathbb Z^k\), and let \(\mu(z)=\sum |z_j|^2a_j\). At \(a=0\), the level set is a cone. The data are real analytic and proper, so Theorem~\ref{thm-singular-tubular} applies and an equivariant tubular neighborhood exists. By Theorem~\ref{thm-tubular-case},
		\[
		\cH_{T^k}^*(0)\cong H_{T^k}^*(\mu^{-1}(0)).
		\]
		At a point \(p\) represented by \(x\), the stalk is \(S(\mathfrak g_x^*)^{G_x}\), a polynomial ring in \(\dim\mathfrak g_x\) generators of degree \(2\). The formal Kirwan map is the classical restriction map and is surjective in even degrees under the fixed-point vanishing hypothesis.
	\end{example}
	
	\begin{example}[Formal vs.\ intersection cohomology vs.\ resolution for \(S^1\ltimes\mathbb C^2\)]\label{ex-all-three}
		For \(S^1\ltimes\mathbb C^2\) with weights \((1,-1)\) and \(a=0\), the intersection cohomology and the resolution cohomology are
		\[
		IH^*(X)\cong\mathbb R\oplus\mathbb R[-2],\qquad
		H^*(\widetilde{\mathscr{M}}_0)\cong\mathbb R\oplus\mathbb R[-2]\oplus\mathbb R[-4].
		\]
		By the previous example, the formal cohomology is
		\[
		\cH_{S^1}^*(0)\cong H_{S^1}^*(\mu^{-1}(0))\cong\mathbb R[u],
		\]
		which is infinite-dimensional. The Kiem--Woolf map to \(IH^*(X)\) is not known to be surjective in general; the DRS map to \(H^*(\widetilde{\mathscr{M}}_0)\) is weakly surjective but not surjective; the formal map to \(\cH_{S^1}^*(0)\) is the classical restriction map and is surjective. The formal target contains the Borel cohomology and is therefore larger than the resolution target in this example; its advantage is that it is defined without any resolution procedure.
	\end{example}
	
	The examples establish the following points.
	
	\begin{enumerate}[label=(\roman*)]
		\item When an equivariant tubular neighborhood exists—for instance under the real analytic/proper or Whitney stratified assumptions of Theorem~\ref{thm-singular-tubular}—the formal cohomology is canonically isomorphic to the Borel equivariant cohomology of the level set, and the formal Kirwan map is the classical restriction map.
		\item The local equivariant homotopy retraction condition is strictly weaker than the existence of a global equivariant tubular neighborhood. It guarantees that the formal cohomology is well-defined and independent of the choice of neighborhood, and it yields surjectivity of the formal Kirwan map under additional global hypotheses.
		\item In the singular cone case, the formal framework recovers the DRS even-degree surjectivity theorem while avoiding the resolution construction.
	\end{enumerate}

	\section{Open problems}\label{sec-open}
	
	We conclude with several open problems.
	\begin{enumerate}[label=(\arabic*)]
		\item Characterize the singular values \(a\) for which the local equivariant homotopy retraction condition holds but no global equivariant tubular neighborhood exists, and compute the resulting formal cohomology.
		\item Determine whether the formal Kirwan map is surjective in odd degrees under additional hypotheses on the slice level sets, or whether the odd-degree failure is unavoidable.
		\item Compare the formal cohomology with the intersection cohomology of the singular quotient in cases where both are defined.
		\item Extend the formal framework to non-compact groups and to Hamiltonian actions with proper moment maps.
	\end{enumerate}

	\section{Tool Use}
	
	During the preparation of this manuscript, the author used ChatGPT (version 6.0) to check the textual aspects of this manuscript, and made corrections to any errors and inappropriate expressions. The author consulted ChatGPT for suggestions on improving the manuscript and made revisions. The author has reviewed and edited the output and takes full responsibility for the content of this publication.

\end{document}